\documentclass[11pt]{article}
\usepackage[margin=1.1in]{geometry}
\usepackage{amsmath,amssymb,amsthm}
\usepackage[hidelinks]{hyperref}

\newtheorem{theorem}{Theorem}
\newtheorem{proposition}{Proposition}

\theoremstyle{remark}
\newtheorem{remark}{Remark}

\title{An explicit construction of two completely independent\\
spanning trees in the four-dimensional dual-cube}
\author{Jitendra Prajapati\\ \small Independent Researcher\\ \small \href{mailto:jitendraprajapati2603@gmail.com}{\texttt{jitendraprajapati2603@gmail.com}}}
\date{August 2026}

\begin{document}
\maketitle

\begin{abstract}
Lalou, Mbarek, Skender and Togni (arXiv:2607.25917) proved that the
$n$-dimensional dual-cube $F_n$ admits two completely independent spanning
trees for every $n\ge 5$, observed that none exist for $n\le 3$, and
identified $F_4$ as the first unresolved case, reporting more than $700$
hours of inconclusive computation. We settle this case affirmatively by an
explicit construction, completing the classification: $F_n$ admits two
completely independent spanning trees if and only if $n\ge 4$. The
internal-vertex sets of the two trees are the level sets of a single
ten-term cubic polynomial over $\mathbb{F}_2$ in the seven vertex bits, and
correctness reduces to finite connectivity checks that are machine-verified
by a solver-free program distributed with the certificate. In $F_4$ the two
trees necessarily use $254$ of the $256$ edges. We also report exact
infeasibility results for simpler rules of the same shape: within the
search model, no affine or quadratic rule works, and ten terms is the
fewest possible for a cubic rule.
\end{abstract}

\section{Introduction}

Let $G=(V,E)$ be a simple undirected graph. A spanning tree $T$ of $G$ is a
connected acyclic subgraph with $V(T)=V(G)$; a vertex $x$ is an
\emph{internal} vertex of $T$ if $d_T(x)\ge 2$ and a \emph{leaf} otherwise.
Two $(x,y)$-paths $P_1,P_2$ are \emph{openly disjoint} if they are
edge-disjoint and share no vertex other than $x$ and $y$. Spanning trees
$T_1,\dots,T_k$ of $G$ are \emph{completely independent} (CIST) if for
every pair of vertices $x,y$, the $(x,y)$-paths in $T_1,\dots,T_k$ are
pairwise openly disjoint.

The study of CISTs began with Hasunuma~\cite{Hasunuma2001}, who gave the
following characterization.

\begin{theorem}[\cite{Hasunuma2001}]\label{thm:hasunuma}
Spanning trees $T_1,\dots,T_k$ of a graph $G$ are completely independent if
and only if they are edge-disjoint and, for every vertex $x\in V(G)$, there
is at most one index $i$ with $d_{T_i}(x)>1$.
\end{theorem}

Hasunuma also proved that deciding the existence of two CISTs is
NP-complete~\cite{Hasunuma2002}. Araki~\cite{Araki2014} reformulated
existence in terms of vertex partitions; we use the following form, as
restated in~\cite{QinHaoWu2022}.

\begin{theorem}[\cite{Araki2014}]\label{thm:araki}
A connected graph $G$ has two completely independent spanning trees if and
only if $V(G)$ admits a partition $\{V_1,V_2\}$ such that $G[V_1]$ and
$G[V_2]$ are connected and the bipartite subgraph of $G$ induced by the
edge set $E(V_1,V_2)$ has no tree component. In that case the trees can be
chosen so that $V_i$ is exactly the set of internal vertices of $T_i$.
\end{theorem}

Constructions of two CISTs are known for hypercubes and several dense
hypercube variants~\cite{PaiChang2016}; see the survey~\cite{ChengWangFan2023}.
The dual-cube $F_n$, introduced by Li and Peng~\cite{LiPeng2000}, is sparse
by design, which places it near the edge-count threshold for two
edge-disjoint spanning trees. Lalou, Mbarek, Skender and
Togni~\cite{Lalou2026} recently proved that $F_n$ admits two CISTs for
every $n\ge 5$, observed that the edge count already obstructs existence
for $n\le 3$, and reported that the case $n=4$ resisted more than $700$
hours of computation and was left unresolved.

In this note we settle the remaining case.

\begin{theorem}\label{thm:main}
$F_4$ admits two completely independent spanning trees. Consequently, the
dual-cube $F_n$ admits two completely independent spanning trees if and
only if $n\ge 4$.
\end{theorem}

The construction (Section~\ref{sec:construction}) is compact: the
internal-vertex partition is given by a ten-term cubic polynomial over
$\mathbb{F}_2$ in the seven vertex bits, and its correctness reduces via
Theorem~\ref{thm:araki} to finite connectivity checks. A certificate
pinning the two trees explicitly, together with a verifier that uses only
the Python standard library and checks all $\binom{128}{2}=8128$ pairs of
tree paths directly against the definition, is publicly available
(Section~\ref{sec:verification}). Section~\ref{sec:searches} reports exact
infeasibility results for simpler rules of the same shape.

\section{Preliminaries}\label{sec:prelim}

Following~\cite{Lalou2026}, the dual-cube $F_n$ is the $n$-regular
$n$-connected graph on the vertex set $\{0,1\}^{2n-1}$ in which two
vertices $u=(u_{2n-1},\dots,u_1)$ and $v=(v_{2n-1},\dots,v_1)$ are adjacent
if and only if they differ in exactly one bit position $i$ and:
\begin{enumerate}
\item if $1\le i\le n-1$ then $u_{2n-1}=v_{2n-1}=0$;
\item if $n\le i\le 2n-2$ then $u_{2n-1}=v_{2n-1}=1$;
\item $i=2n-1$ is always allowed.
\end{enumerate}
Thus $F_4$ has $128$ vertices, $256$ edges, and degree $4$; it consists of
sixteen $Q_3$ clusters, eight in each class, joined by a perfect matching
of $128$ cross-edges.

For consistency with the machine-readable certificate we index bits from
zero: a vertex $v\in\{0,1\}^7$ of $F_4$ has bits $x_0,\dots,x_6$ with
$x_i=u_{i+1}$, so $x_6$ is the class bit, $x_0,x_1,x_2$ are the class-$0$
cluster coordinates and $x_3,x_4,x_5$ are the class-$1$ cluster
coordinates.

Two edge-disjoint spanning trees of $F_4$ use $2\cdot 127=254$ of the $256$
edges. By Theorem~\ref{thm:hasunuma}, in any CIST pair no vertex is
internal in both trees; since $F_4$ is $4$-regular, every vertex is then
either internal in exactly one tree or a leaf in both.

\section{The construction}\label{sec:construction}

Define $z\colon V(F_4)\to\mathbb{F}_2$ by
\begin{equation}\label{eq:anf}
z(v) \;=\; x_0x_1 \oplus x_2 \oplus x_1x_2 \oplus x_1x_3 \oplus x_3x_4
\oplus x_1x_5 \oplus x_4x_5 \oplus x_0x_6 \oplus x_3x_6 \oplus x_2x_4x_6,
\end{equation}
and set $V_1=\{v : z(v)=1\}$ and $V_2=\{v : z(v)=0\}$.

\begin{proposition}\label{prop:partition}
$|V_1|=|V_2|=64$; the induced subgraphs $F_4[V_1]$ and $F_4[V_2]$ are
connected with exactly $64$ edges each (hence unicyclic); and the bipartite
subgraph induced by the $128$ edges of $E(V_1,V_2)$ has exactly two
components, each with $64$ vertices and $64$ edges (hence unicyclic).
\end{proposition}

\begin{proof}
This is a finite computation on a $128$-vertex graph; it is performed
independently of any solver by the verification program of
Section~\ref{sec:verification}.
\end{proof}

Since a unicyclic component is not a tree, Proposition~\ref{prop:partition}
and Theorem~\ref{thm:araki} prove Theorem~\ref{thm:main}: $\{V_1,V_2\}$ is
a CIST partition of $F_4$, and the classification follows
from~\cite{Lalou2026}. Every vertex is internal in exactly one tree, so the
pair realizes the balanced $64/64$ pattern.

The certificate pins the two trees explicitly, so correctness does not
depend on quoting Theorem~\ref{thm:araki}. The unique cycle of $F_4[V_2]$
contains the edge $\{0,64\}$ and the unique cycle of $F_4[V_1]$ contains
the edge $\{21,85\}$. Orient each unicyclic component of the cut graph
functionally, so that every vertex acquires exactly one outgoing cut edge
$\{u,\mathrm{out}(u)\}$. Then
\[
T_2^{\,\prime} = \bigl(E(F_4[V_2])\setminus\{\{0,64\}\}\bigr)\;\cup\;
\bigl\{\{u,\mathrm{out}(u)\} : u\in V_1\bigr\},
\]
and symmetrically $T_1^{\,\prime}$ with $\{21,85\}$ removed, are two
edge-disjoint spanning trees, with internal-vertex sets exactly $V_1$ and
$V_2$; the only edges of $F_4$ unused by $T_1^{\,\prime}\cup
T_2^{\,\prime}$ are $\{0,64\}$ and $\{21,85\}$.

\section{Machine verification}\label{sec:verification}

The construction is distributed as a short JSON certificate together with a
verifier that uses only the Python standard library --- no solver is
involved. The verifier (i)~rebuilds $F_4$ from the bit definition of
Section~\ref{sec:prelim}, (ii)~evaluates the polynomial~\eqref{eq:anf} and
checks every statement of Proposition~\ref{prop:partition},
(iii)~reconstructs $T_1^{\,\prime}$ and $T_2^{\,\prime}$ and checks that
they are edge-disjoint spanning trees with internal-vertex sets $V_1$ and
$V_2$, and (iv)~for each of the $8128$ unordered vertex pairs checks
directly that the two tree paths are openly disjoint. Step~(iv) is the
definition of complete independence, so the verified claim does not rest on
Theorem~\ref{thm:araki} or on any search software. A second, independently
fixed edge-list witness passes a separate verifier. Both programs refuse to
run under Python's assertion-stripping \texttt{-O} mode.

\section{Exact searches for simpler rules}\label{sec:searches}

The rule~\eqref{eq:anf} was found with a CP-SAT model (OR-Tools) whose
Boolean variables encode two edge-disjoint spanning trees together with the
coefficients of an algebraic normal form (ANF) for $z$, subject to the
\emph{exact-complement} pattern: every vertex is internal in exactly one
tree, with $64$ internal vertices in each. This pattern is a genuine model
restriction --- a general CIST pair may contain vertices that are leaves in
both trees --- and the results below are statements about this model only.
Within it, the solver returned exact infeasibility (not timeouts) for:
\begin{itemize}
\item all $127$ nonconstant affine rules
$z(v)=\mathrm{parity}(v\wedge m)$, $m\in\{1,\dots,127\}$;
\item every ANF of degree at most $2$ --- so, with the degree-$1$ case,
degree $3$ is the minimum ANF degree realizable in the model;
\item every ANF of degree at most $3$ with at most nine nonconstant terms
--- so ten terms is the minimum term count among degree-at-most-$3$ rules
in the model (sparser rules of degree exceeding $3$ were not excluded);
\item all rules of the separable form $f(x_0,x_1,x_2)\oplus
g(x_3,x_4,x_5)\oplus\alpha x_6$, and eight relation-based rules
$h(\mathrm{op}(A,B))\oplus\alpha x_6$ for $A=(x_0,x_1,x_2)$,
$B=(x_3,x_4,x_5)$, $\mathrm{op}\in\{\oplus,\,+\bmod 8,\,-\bmod 8,\,
\mathrm{Gray}(A)\oplus\mathrm{Gray}(B)\}$ and $\alpha\in\{0,1\}$.
\end{itemize}
These are solver statuses for the stated finite models. CP-SAT does not
emit solver-independent unsatisfiability certificates, so the negative
results carry the solver's usual trust assumptions (solver version and run
parameters are recorded with the data); the positive result of
Theorem~\ref{thm:main} does not depend on them.

\begin{remark}\label{rem:compound}
$F_4$ has the compound form $K_{8,8}(Q_3)$: contracting its sixteen $Q_3$
clusters yields $K_{8,8}$. CIST results for compound
graphs~\cite{QinHaoChang2020} state corollaries for compound graphs with
complete-graph clusters, as in $L\text{-}HSDC_m(m)=Q_m(K_m)$, and the
constructions in this line of work proceed via edge-disjoint Hamiltonian
cycles with clusters $K_{2n}$~\cite{ChengWangFan2023}. We have not been
able to consult the full sufficient condition of~\cite{QinHaoChang2020};
we note, however, that $Q_3$ is not complete and its $12$ edges cannot
carry two edge-disjoint spanning trees on $8$ vertices, so approaches
requiring per-cluster tree pairs cannot apply to $F_4$.
\end{remark}

\section{Concluding remarks}

The polynomial~\eqref{eq:anf} answers the existence question for $F_4$ with
a certificate short enough to print. Two natural questions remain. First,
the minimality statements of Section~\ref{sec:searches} are relative to the
exact-complement model and to the solver; a solver-independent proof that
no quadratic rule works, or a simpler rule outside the model (allowing
vertices that are leaves in both trees), would be of interest. Second,
Lalou et al.~\cite{Lalou2026} conjecture on the existence of $k$ CISTs in
higher-dimensional dual-cubes; whether algebraically defined partitions of
the present kind can produce three or more CISTs in $F_n$ for larger $n$ is
open.

\section*{Data availability}

The certificate, both verifiers, and the exact search records (including
solver versions and run parameters) are available at
\url{https://github.com/infinityscroll/f4-dualcube-cist}.
Verification requires one command:
\texttt{python3 verify\_anf\_structure.py anf-construction.json}.

\section*{Acknowledgements}

The search programs, the certificate, and an initial draft of this note
were produced with AI assistance (OpenAI Codex, with subsequent
verification and preparation assisted by Claude); the author directed the
work, independently re-verified the construction, and takes full
responsibility for the content. A documented literature search located no
earlier construction of two completely independent spanning trees in
$F_4$; no claim of priority is made pending expert review.

\end{document}